\documentclass[a4paper,12pt]{amsart}
\usepackage[utf8]{inputenc}  
\usepackage[T1]{fontenc}
\usepackage[english]{babel}

\usepackage{amsthm}
\usepackage{amsmath}
\usepackage{amssymb}
\usepackage{mathrsfs}
\usepackage{dsfont}

\usepackage{enumitem}
\usepackage{ifthen}
\usepackage{verbatim}
\usepackage{geometry}
\usepackage{hyperref}
\usepackage{url}

\newcommand{\R}{{\mathbb R}}
\newcommand{\N}{{\mathbb N}}
\newcommand{\Z}{{\mathbb Z}}
\newcommand{\Q}{{\mathbb Q}}
\newcommand{\C}{{\mathbb C}}

\usepackage{mathtools}

\newcommand*{\tttt}[2][-3mu]{\ensuremath{\mskip1mu\prescript{\smash{\mathrm t\mkern#1}}{}{\mathstrut#2}}}%

\DeclareMathOperator{\Imag}{Im}
\DeclareMathOperator{\Real}{Re}
\DeclareMathOperator{\Ker}{Ker}
\DeclareMathOperator{\Sym}{Sym}
\DeclareMathOperator{\Skew}{Skew}
\DeclareMathOperator{\Id}{Id}

\newtheorem{prop}{Proposition}[section]
\newtheorem{thm}[prop]{Theorem}
\newtheorem{cor}[prop]{Corollary}
\newtheorem{lem}[prop]{Lemma}
\newtheorem*{defi}{Definition}

\title{Quantum limits for harmonic oscillator}
\author{Elie Studnia}

\begin{document}

\maketitle

\tableofcontents

\textbf{Acknowledgements}
This work was written during a visit at UC Berkeley, under the direction of S. Dyatlov and M. Zworski. We thank this institution and professor Dyatlov and Zworski for their help, suggestions and comments.

\newpage

\section*{Introduction}

In this note we consider high energy eigenfunctions of the harmonic oscillator in $\R^d$ and prove that any invariant measure on the energy surface can be written as a weak limit of eigenfunctions. 

To motivate our result, let us recall some facts about the positive Laplacian $\Delta_g$ on a compact Riemannian manifold $(M,g)$. In \cite{jac-zel} it is for instance shown that when $(M,g) = S^n$ with its canonical metric, any invariant measure on $S^*M$ is a weak limit of eigenfunctions. 

In \cite{zel-zw} it is shown that if the geodesic flow on $S^*M$ is ergodic, quantum ergodicity for the eigenfunctions of the Laplacian occurs (in particular, there is a density one subsequence of these that go to the uniform probability measure).

\smallskip
\noindent

{\bf Theorem}{\em \,\, Let $d \geq 1$ be an integer, $\mu$ be a probability measure over $S^{2d-1}$ that is invariant under the Hamiltonian flow of $p: (x,\xi) \in \R^{2d} \longmapsto |x|^2+|\xi|^2$. Then, there exists sequences $(h_k)$ of positive real numbers, $(u_k)$ of smooth $L^2(\R^d)$ functions such that
\begin{enumerate}
\item $(h_k) \rightarrow 0$
\item $\forall k,\, (-h_k^2\Delta+|x^2|)u_k = u_k$
\item $\|u_k\|_{L^2} = 1$
\item $$\forall a \in \mathscr{C}^{\infty}_c(\R^{2d}),\, \langle a^w(x,h_kD)u_k\,,u_k \rangle \longrightarrow \int_{S^{2d-1}}{ad\mu}$$
\end{enumerate}}

\smallskip
\noindent

This theorem is inspired by a result of Jakobson-Zelditch \cite{jac-zel}, asserting that any probability measure over $S^*M$, where $M$ is the $d$-dimension sphere, endowed with its canonical Riemannian structure, is a weak-* limit of Wigner measures associated with eigenfunctions of the Laplacian if, and only if, it is invariant under the geodesic flow. A similar result to ours was also proved in \cite[Prop. 5]{OVVB} and generalized in \cite[Th 1.7]{vic-arn}. Related results on-self adjoint perturbations of the harmonic oscillator were also studied in \cite{var-gri}.

\smallskip
\noindent
To a classical observable $p = p(x,\xi)$ we associate the Hamilton vector field $H_p(x,\xi) = (\partial_{\xi}p(x,\xi),-\partial_xp(x,\xi))$, and the related flow. In our case, $p(x,\xi) = |x|^2+|\xi|^2$ and $H_p(x,\xi) = (2\xi,-2x)$. We see that if $z=(x,\xi) \in S^{2d-1}$, $(2\xi,-2x) \in T_zS^{2d-1}$, therefore $S^{2d-1}$ is invariant under the flow. 
It is easy to check that the trajectory starting from $(x_0,\xi_0)$ has the parametric equation:
$$x(t) = (\cos{2t})x_0+(\sin{2t})\xi_0,\,\xi(t) = (\cos{2t})\xi_0-(\sin{2t})x_0\,,$$
\noindent
therefore the trajectories of the flow in $S^{2d-1}$ are disjoint great circles (this will be important later on), which we will call (for the sake of simplicity) the orbits of $S^{2d-1}$.

\medskip

We adopt the strategy of \cite{jac-zel} to our setting and proceed as follows:
\smallskip
\begin{enumerate}
\item Prove the statement when $\mu$ is the uniform measure over one particular closed orbit, by using the $d=1$ case.
\item Use a group action to construct a family of eigenfunctions concentrating on any given closed orbit (that is, any uniform measure on any closed orbit of the flow; we will call them elementary measures). 
\item Prove that mixed terms from different orbits in step $2$ are negligible, and deduce from this the statement when $\mu$ is a convex combination of elementary measures (which we will call convex measures).
\item Show that for in suitable topology, the convex measures are dense in the set of all probability measues over $S^{2d-1}$ invariant under the flow, and that the set of measures that can be obtained in the theorem is closed. 
\end{enumerate}
\medskip
\noindent

Section $1$ gives preliminary results. Section $2$ proves the case $d=1$ and completes step $1$. Section $3$ is about finding the right group action to make step $2$ work. In section $4$ we complete the proof.

\newpage 

\section{Preliminaries}

\subsection{Review of semiclassical analysis}

We use the notations and results of \cite{ev-zw}, mostly chapters $4$, $5$ and $11$. The most important result we need throughout this note is \cite[Theorem 4.23]{ev-zw}:

\begin{thm}
\label{bounded}
Let $a \in S(1_{\R^{2d}})$. Then, for any $h > 0$, $a^w(x,hD)$ is a bounded operator from $L^2(\R^d)$ to itself, and we have, for some constants $C=C_d$, $M$ not depending on $d$:
$$\|a^w(x,hD)\|_{L^2 \rightarrow L^2} \leq C\sum_{|\alpha| \leq Md}{h^{|\alpha|/2}\|\partial_{\alpha}a\|_{\infty}}.$$
\end{thm}

Another very basic estimate is what happens when we compose operators associated with $S(1)$ symbols (\cite[Theorem 4.18]{ev-zw}):

\begin{thm}
\label{compo_S1}
Let $a,b \in S(1_{\R^{2d}})$. Then $a^w(x,hD)b^w(x,hD) = c^w(x,hD)$, for some $c \in S(1)$ such that $c=ab+O_{S(1)}(h)$, and the estimates are explicit. In particular, $[a^w(x,hD),b^w(x,hD)] = O_{L^2 \rightarrow L^2}(h)$.
\end{thm}

\subsection{Microlocalisation near the sphere}
Let $d \geq 1$ be an integer.
 
\begin{prop}
\label{micro}
Let $h_n$ be a sequence going to $0$, with $$h_n = \frac{1}{2n+d}+o(n^{-1}).$$ Let, for each $n \geq 0$, $g_n \in \mathscr{S}(\R^d)$ be such that $$(-h_n^2\Delta+|x|^2)g_n=\alpha_ng_n,$$ where $\alpha_n \longrightarrow 1$, and $\|g_n\|_{L^2} = 1$. Let $a \in S(1)$ be zero near $\{|x|^2+|\xi|^2=1\}$. Then, explicitly, $\|a^w(x,h_nD)g_n\|_{L^2}=o(1)$. 
\end{prop}

\begin{proof}
There exists some $0 < r < 1$ such that $a = 0$ on $\{1-r < |x|^2+|\xi|^2 < 1+r \}$. Let $\chi_0$ be a smooth compactly supported function from $\R$ to $\R$ such that $\chi_0$ is supported inside $(-r,r)$ and is $1$ near $0$. Let $\chi(x,\xi) = \chi_0(|x|^2+|\xi|^2-1)$, so that $a$ and $\chi$ have disjoint supports. 

Let $p_1(x,\xi) = |x|^2+|\xi|^2-1+i\chi(x,\xi)$.
Notice that $p_1^w(x,h_nD)g_n = i\chi^w(x,h_nD)g_n+\beta_ng_n$, where $\beta_n=o(1)$ is a real number.

Since $|p_1(x,\xi)| \geq c(x^2+\xi^2+1)$ for some small $c > 0$, $p_2 := p_1^{-1}$ is a $S((x^2+\xi^2+1)^{-1})$ symbol. Thus $(p_2^w(x,h_nD)p_1^w(x,h_nD)-1)g_n = o_{L^2}(1)$. So $$a^w(x,h_nD)g_n = ia^w(x,h_nD)p_2^w(x,h_nD)\chi^w(x,h_nD)g_n + \beta_na^w(x,h_nD)p_2^w(x,h_nD)g_n + o_{L^2}(1),$$ therefore (the $a^w(x,hD),p_2^w(x,hD)$ are explicitly bounded as $h$ goes to $0$, see \ref{bounded}) $$a^w(x,h_nD)g_n = ia^w(x,h_nD)p_2^w(x,h_nD)\chi^w(x,h_nD)g_n+o_{L^2}(1).$$
From \ref{compo_S1}, and \ref{bounded}, $$a^w(x,hD)p_2^w(x,hD)\chi^w(x,hD)=(ap_2\chi+O_{S(1)}(h))^w(x,hD) = O_{L^2\rightarrow L^2}(h).$$
\end{proof}

\subsection{Possible weak limits}
The following is almost addressed in \cite[Sections 5.1, 5.2]{ev-zw}, but we need one very minor modification. 

\begin{prop}
\label{converse}
Let $h_n = \frac{1}{2n+d}+o(n^{-1})$, $n \in N$, where $N$ is an infinite subset of $\N$. Let, for each $n \in N$, be $g_n \in \mathscr{S}(\R^d)$ with unit $L^2$ norm such that $(-h_n^2\Delta + |x|^2)g_n = \alpha_ng_n$, and $\alpha_n=1+o(1)$. Assume that for every $a \in \mathscr{C}^{\infty}_c(\R^{2d})$, the sequence $\langle a^w(x,h_nD)g_n,\,g_n\rangle$, $n \in N$, converges. Then there exists a probability measure $\mu$ on $S^{2d-1}$ invariant under the Hamiltonian flow of $p$ such that, as $n \in N$ goes to infinity, 
$$\forall a \in S(1),\, \langle a^w(x,h_nD)g_n,\,g_n\rangle \longrightarrow \int{a\,d\mu}.$$
\end{prop}

\begin{proof}
From \ref{bounded} we know that the function $\left(\mathscr{C}^{\infty}_c(\R^{2d}),\|\|_{\infty}\right) \longrightarrow \C$ that maps $a$ to $$\lim \,\langle a^w(x,h_nD)g_n,\,g_n\rangle$$ is a non-negative (because of the sharp G\aa rding inequality, see \cite[Theorem 4.32]{ev-zw}) bounded linear form. Thus, from the Riesz representation theorem, it is $$a \longmapsto \int_{\R^{2d}}{a\,d\mu}$$ for some finite non-negative Radon measure $\mu$ on $\R^{2d}$. Applying \ref{micro} to the $g_n$, we deduce that $\mu$ is supported on $S^{2d-1}$. \\Moreover, let $a \in \mathscr{C}^{\infty}_c(\R^{2d})$ be $1$ near $S^{2d-1}$; we know from \ref{micro} that $(1-a)^w(x,h_nD)g_n = o_{L^2}(1)$, hence $\mu\left(S^{2d-1}\right) = \int{a\,d\mu} = \langle a^w(x,h_nD)g_n,\,g_n\rangle +o(1) = \langle g_n,g_n\rangle + o(1) = 1+o(1)$. So $\mu$ is a probability measure on $S^{2d-1}$. \\
Considering the invariance part, we proceed as in the proof of \cite[Theorem 5.4]{ev-zw}, except that the remainder is not $o(h_n)$. But looking at that proof, it turns out that $o(1)$ is enough provided that we can prove that for any symbol $a \in \mathscr{C}^{\infty}_c(\R^{2d})$, $$\langle [p^w(x,h_nD),a^w(x,h_nD)]g_n,\,g_n \rangle = o(h_n).$$ In this case, 
\begin{align*}
\langle [p^w(x,h_nD),a^w(x,h_nD)]g_n,\,g_n \rangle &= \langle p^w(x,h_nD)a^w(x,h_nD)g_n,\,g_n\rangle - \alpha_n\langle a^w(x,h_nD)g_n,\,g_n\rangle \\
&= \langle a^w(x,h_nD)g_n,\,p^w(x,h_nD)g_n\rangle - \alpha_n\langle a^w(x,h_nD)g_n,\,g_n\rangle\\
&=0,
\end{align*} 
which ends the proof.
\end{proof}

\section{First step}

\subsection{Description of the one-variable eigenfunctions}
~\\
Let $h > 0$, we consider the operator $P = -h^2\Delta + x^2$ (known as the quantum harmonic oscillator), defined over $\mathscr{S}(\R) \subset L^2(\R)$. Let us define the annihilation $A = h\partial + x$ and creation operators $A^*=-h\partial + x$, such that $P = AA^*+h = A^*A-h$. Let $u_0^h(x) = \exp(-\frac{x^2}{2h})$, and $u_n^h = \left(A^*\right)^n(u_0^h) \in \mathscr{S}(\R)$. \\
\smallskip
\noindent
An easy induction proves that $u_n^h$ is an eigenfunction of $P$ with eigenvalue $(2n+1)h$, and can be written as a polynomial with real coefficients of degree $n$ times $u_0^h$; it is known that (see, for instance \cite[Chapter 6]{ev-zw}) that the $u^n_h$ span a dense vector space. Since $P$ is self-adjoint, we have therefore found all its eigenfunctions. 

\subsection{One-dimension theorem}

Let, for each $n \geq 0, h > 0$, $v_n^h = \|u_n^h\|_{L^2}^{-1}u_n^h$. 
The following is proved using \ref{converse}.
\begin{prop}
Let $a \in \mathscr{C}^{\infty}_c(\R^2)$. Let $h_n = \frac{1}{2n+1}+o(n^{-1})$. Then $$\langle a^w(x,hD)v_n^{h_n},\,v_n^{h_n}\rangle \longrightarrow \frac{1}{2\pi}\int_{x^2+\xi^2=1}{a\,ds}.$$
\end{prop}
\begin{proof}
From \ref{bounded}, we know that the sequence $\langle a^w(x,hD)v_n^{h_n},\,v_n^{h_n}\rangle$ is bounded. So we just have to prove that any limit point of it is $$\frac{1}{2\pi}\int_{x^2+\xi^2=1}{a\,ds}.$$
So taking subsequences, assume $\langle a^w(x,hD)v_{n_k}^{h_{n_k}},\,v_{n_k}^{h_{n_k}}\rangle$ converges to $\ell$. From \cite[Theorem 5.2]{ev-zw} we can take a subsequence of the $n_k$ (which, for the sake of simplicity, we will also name $n_k$), such that there exists a non-negative Radon measure $\mu$ on $\R^2$, 
$$\forall b \in \mathscr{C}^{\infty}_c,\, \langle b^w(x,h_{n_k}D)v_{n_k}^{h_{n_k}},\,v_{n_k}^{h_{n_k}}\rangle \underset{k\rightarrow\infty}{\longrightarrow} \int{b\,d\mu}, $$ and we also have $$ \int{a\,d\mu} = \ell.$$
From \ref{converse}, $\mu$ is a probability measure on $S^1$ that is invariant under the Hamilton flow of $x^2+\xi^2$: so it must be the uniform measure.
\end{proof}

\subsection{One specific orbit in dimension $d$}

To go from the one-dimension case to the higher-dimension, an easy way uses the following technical result, proved for instance in \cite[Theorem 39.2]{Trv} (with a slight modification in the argument to allow more factors in the tensor product). 

\begin{prop}
\label{dense}
Let $f \in \mathscr{C}^{\infty}_c(\R^d)$. There exists a sequence $g_n$ of $\mathscr{C}^{\infty}_c(\R^d)$ functions that are linear combinations of tensor products of $\mathscr{C}^{\infty}_c(\R)$ functions such that $g_n \underset{\mathscr{C}^{\infty}_c(\R^d)}{\longrightarrow} f$. 
\end{prop}

Let, here, $d \geq 2$ be an integer, $h_n = \frac{1}{2n+d} = \frac{1}{2n+1} + o\left(\frac{1}{n}\right)$, and $f_n = v^{h_n}_n \otimes v^{h_n}_0 \otimes \ldots \otimes v^{h_n}_0$, so that $f_n$ is a $d$-variable Schwartz function with unit $L^2$ norm, and $(-h_n^2\Delta + |x|^2)f_n = f_n$. 

The following lemma is exactly the first example following \cite[Theorem 5.2]{ev-zw}, for the coherent state $(0,0)$. 

\begin{lem}
Let $a \in \mathscr{C}^{\infty}_c(\R^2)$, then, as $n \longrightarrow \infty$, 
$$\langle a^w(x,h_nD)v^{h_n}_0,\,v^{h_n}_0\rangle \longrightarrow a(0,0)$$
\end{lem}

Now we may generalise:

\begin{prop}
\label{oneorbitdim}
Let $a \in \mathscr{C}^{\infty}_c(\R^{2d})$. Then, as $n \longrightarrow \infty$:
$$\langle a^w(x,h_nD)f_n,\,f_n \rangle \longrightarrow \frac{1}{2\pi}\int_{x_1^2+\xi_1^2=1}{a\,ds}.$$
\end{prop}

\begin{proof}
We know that for any $n$, from \cite[Theorem 4.23]{ev-zw}, for some universal constants $C,M \geq 1$, $$|\langle a^w(x,h_nD)f_n,\,f_n\rangle|\leq C\|a\|_{\mathscr{C}^{Md}_b}.$$
So it is enough to prove the convergence on some $\|\cdot\|_{\mathscr{C}^{Md}_b}$-dense subset of $\mathscr{C}^{\infty}_c$. Since the statement is linear in $a$, it is enough to prove it on a set of $a$ that spans a $\|\cdot\|_{\mathscr{C}^{Md}_b}$-dense subspace of $\mathscr{C}^{\infty}_c$. Using \ref{dense}, it is enough to prove it when $a(x,\xi) = b_1(x_1)c_1(\xi_1)\ldots b_d(x_d)c_d(\xi_d)$. 

We then have, using the previous lemma and section $1.1$:
\begin{align*}
\langle a^w(x,h_nD)f_n,\,f_n\rangle &= \langle (b_1 \otimes c_1)^w(x,h_nD)v_n^{h_n},\,v_n^{h_n}\rangle\prod_{i=2}^d{\langle(b_i\otimes c_i)^w(x,h_nD)v^{h_n}_0,\,v^{h_n}_0\rangle}\\
&\longrightarrow \frac{1}{2\pi}\left(\prod_{i=2}^d{(b_i\otimes c_i)(0,0)}\right)\int_{x^2+\xi^2=1}{(b_1\otimes c_1)(x,\xi)\,ds}\\
&= \frac{1}{2\pi}\int_{\substack{x_1^2+\xi_1^2=1\\|x|^2+|\xi|^2=1}}{a\,ds}
\end{align*}
\end{proof}

\section{Symplectic geometry on the sphere}

The purpose of this section is to show that any closed orbit under the flow of $H_p$ can be mapped to any other orbit by a symplectic orthogonal linear transformation. Quantizations of such transformations are possible (because they are linear symplectic) and map eigenfunctions to eigenfunctions, because they leave the operator $-h^2\Delta+|x|^2$ invariant (they are orthogonal).  

In $3.1$ and $3.2$, $n \geq 1$ is an integer. We denote, for every integer $k \geq 1$, 
 
\begin{itemize}
\item $S_k(\R)$ the set of all $k \times k$ symmetric matrices with real coefficients (real dimension is $k(k+1)/2$).
\item $AS_{k}(\R)$ the set of all $p \times p$ skew-symmetric matrices with real coefficients (real dimension is $k(k-1)/2$).
\item $\tttt{A}$ the transpose of $A$, where $A$ is a $k \times k$ matrix.
\item $\mathscr{M}_k(\R)$ (resp. $\mathscr{M}_k(\C)$) the set of all $k \times k$ matrices with real (resp. complex) coefficients (real [resp. complex] dimension is $k^2$).
\item $\mathscr{H}_k$ the set of all $k \times k$ Hermitian matrices (real dimension is $k^2$, it is not naturally a complex vector space). 
\item $\Sym{M} = M+\tttt{M}$ for every $k \times k$ matrix $M$.
\item $\Skew{M} = M-\tttt{M}$ for every $k \times k$ matrix $M$.
\end{itemize}

\subsection{Ortho-symplectic group}
~\\~\\
\begin{defi}
The ortho-symplectic group $\Gamma_n$ is the set of all $2n \times 2n$ matrices with real coefficients that are both orthogonal and symplectic. 
\end{defi}

We will need strong properties about $\Gamma_n$.

\begin{prop}
$\Gamma_n$ is a compact Lie subgroup of $GL_{2n}(\R)$. It has dimension $n^2$.
Moreover, $$\Gamma_n = \left\{\mathcal{M}_{A,B}:=\begin{bmatrix}A & B \\ -B & A\end{bmatrix}, A\tttt{A}+B\tttt{B} = I_n, A\tttt{B} \in S_n(\R)\right\},$$ and $$T_{\mathcal{M}_{A,B}}\Gamma_n = \left\{\mathcal{M}_{M,N},M,N \in \mathscr{M}_n(\R), A\tttt{N}-B\tttt{M} \in S_n(\R), A\tttt{M}+B\tttt{N} \in AS_n(\R)\right\}.$$
\end{prop}

\begin{proof}
The compactness and group properties are consequences of the definition. The characterization is proved with easy calculations.  
Regarding the tangent space, let 
$$f: \begin{bmatrix}A && B \\ C && D\end{bmatrix} \in \mathscr{M}_{2n}(\R) \longmapsto (A-D,B+C,A\tttt{A}+B\tttt{B},A\tttt{B}-B\tttt{A}) \in \mathscr{M}_n(\R)^2 \times S_n(\R) \times AS_n(\R).$$

$f$ is smooth, and $\Gamma_n = f^{-1}(\left\{(O_n,O_n,I_n,O_n)\right\})$. To end the proof, it is enough to show that $f$ is a submersion at every point of $\Gamma_n$, that the kernel of $f'$ matches what we claimed was the tangent space at every point of $\Gamma_n$, and that dimensions match.

As far as dimensions are concerned, the domain is a $4n^2$-dimensional manifold, and we are looking at the pre-image of a point in a $3n^2$-dimension manifold, thus if $f$ is a submersion at each point of $\Gamma_n$, then $\Gamma_n$ has dimension $n^2$.  

Now, let $P = \begin{bmatrix}A & B\\-B & A\end{bmatrix} \in \Gamma_n$, let $U,V,W,X \in \mathscr{M}_{n}(\R)$. Note that 
\begin{align*}f'(P)\cdot\begin{bmatrix}U & V \\ W & X\end{bmatrix} &= (U - X, V + W, A\tttt{U}+U\tttt{A}+B\tttt{V}+V\tttt{B},\\& A\tttt{V}+U\tttt{B}-V\tttt{A}-B\tttt{U}) \\&= (U-X,V+W,\Sym(A\tttt{U}+B\tttt{V}),\Skew(A\tttt{V}-B\tttt{U})),\end{align*} therefore 

\begin{align*}
\begin{bmatrix}U & V \\ W & X\end{bmatrix} \in \Ker f'(P) &\Leftrightarrow U = X, V = -W, A\tttt{U}+U\tttt{A}+B\tttt{V}+V\tttt{B} = 0,\\& A\tttt{V}+U\tttt{B}-V\tttt{A}-B\tttt{U} = 0 \\
&\Leftrightarrow U = X, V = -W, A\tttt{U}+B\tttt{V} \in AS_n(\R), A\tttt{V}-B\tttt{U} \in S_n(\R)\\
\end{align*}

Since this matches the depiction of the claimed tangent space of $\Gamma_n$ at $P$, it is enough to prove that $f'(P)$ is surjective. Let now $(U', V', W', X') \in \mathscr{M}_n(\R)^2 \times S_n(\R) \times AS_n(\R)$. Let $U,V \in \mathscr{M}_n(\R)$ be such that $P\begin{bmatrix}\tttt{U} \\ \tttt{V}\end{bmatrix} = \frac{1}{2}\begin{bmatrix}W' \\ X'\end{bmatrix}$; it is possible because $P$ is nonsingular. Let $W = V'-V$, $X = U-U'$; one easily sees then that $f'(P)\cdot\begin{bmatrix}U & V \\ W & X\end{bmatrix} = (U',V',W',X')$.

\end{proof}

\subsection{Action of the ortho-symplectic group over the sphere}
~\\
The goal of the section is to prove the following proposition, which is one key element of the second step of the proof:

\begin{prop}
\label{ortho_symp_trans}
Let $\Gamma_n$ act on $S^{2n-1}$. This action is transitive. Also, $\Gamma_n$ acts on the orbits of $S^{2n-1}$ and this action is transitive.
\end{prop}

We first need a linear algebra lemma. 

\begin{lem}
Let $u \in \C^q$ be a non-zero vector, $q \geq 2$. Then $K_u = \left\{M \in \mathscr{H}_q,Mu = 0\right\}$ is a real vector space of dimension $(q-1)^2$.
\end{lem}

\begin{proof}
Assume, without loss of generality, that $u_q \neq 0$. Let $\Phi(M \in K_u) = \left([M]_{i,j}\right)_{1 \leq i,j < q} \in \mathscr{H}_{q-1}$, $\Phi$ is a linear operator. It is easily checked to be injective and surjective, with an explicit inverse map. 
\end{proof}

\begin{cor}
Let, for any $(x,y) \in S^{2n-1}$, $A_{x,y}$ be the map \begin{align*}A_{x,y}: g \in \Gamma_n \longmapsto g\cdot (x,y) \in S^{2n-1}.\end{align*} $A_{x,y}$ is smooth and is a submersion at $I_{2n}$.
\end{cor}

\begin{proof}
Clearly, $A_{x,y}$ is well-defined (because $\Gamma_n$ is made with orthogonal matrices) and smooth. Let $B_{x,y}(g) = A_{x,y}(g) \in \R^{2n}$, let us prove that the derivative (let us call it $d$) of $B_{x,y}$ at $I_{2n}$ has rank $2n-1$, which will end the proof. 
It is enough to show that its kernel has dimension $(n-1)^2$. 
Now, let us see that $$T_{I_{2n}}(\Gamma_n) = \left\{\begin{bmatrix}M & N \\ -N & M\end{bmatrix},N \in S_n(\R),M \in AS_n(\R)\right\},$$ and that 
$$d\left(\begin{bmatrix}M & N \\ -N & M\end{bmatrix}\right) = (Mx+Ny,-Nx+My) = (\Real((N+iM)(y-ix)),\Imag((N+iM)(y-ix))).$$
Let us denote $z = y-ix \in \C^n \backslash \{0\}$, from the above we deduce that as a real vector space, $\Ker d$ is isomorphic to $\{N+iM,N \in S_n(\R), M \in AS_n(\R),(N+iM)z = 0\} = K_z$. 
\end{proof}

Now, we can prove the proposition \ref{ortho_symp_trans}.

\begin{proof}
Let $g \in \Gamma_n$, let $(x,y) \in S^{2n-1}$. Let $f: h \in \Gamma_n \longmapsto hg$, we see then that $A_{x,y} \circ f = A_{g(x,y)}$. Thus the derivative of $A_{x,y} \circ f$ at $I_{2n}$ is surjective, so the derivative of $A_{x,y}$ at $g = f(I_{2n})$ is surjective too. Therefore $A_{x,y}$ is a submersion. 
Since $A_{x,y}$ is a submersion, it is open and its image is open in $S^{2n-1}$. In addition, $A_{x,y}$ is continuous and defined on a compact topological space, thus its image in $S^{2n-1}$ is compact, hence closed. Since $S^{2n-1}$ is connected, $A_{x,y}$ is onto.

We have seen above that the vector field of which we take the flow is the one given by $J$. So the Hamiltonian flow at time $t$ of $f:z \in \R^{2n} \longmapsto |z|^2$ is the linear isomorphism given by $\exp\left(2tJ\right)$. Since $\Gamma_n$ commutes with $J$, it also commutes with $\exp(sJ)$ for any real number $s$, thus $\Gamma_n$ maps orbits one to another. Since it can map any point of $S^{2n-1}$ to any other, and since there is only one orbit going through a given point, the rest of the statement holds.
\end{proof}

\subsection{Elementary measures are limits of eigenfunctions}

Recall that we defined elementary measures as uniform probability measures that are supported on one single orbit of the flow. 

Recall that from \cite[Theorem 11.9]{ev-zw}, for every symplectic linear operator $A \in \mathcal{L}(\R^{2d})$, and for every $h > 0$, there exists a unitary operator $T_{A,h} \in \mathcal{L}(L^2(\R^d))$ such that for every $a \in S(x^2+\xi^2+1)$, $(a \circ A)^w(x,hD) = T_{A,h}^*a^w(x,hD)T_{A,h}$. 

We recall that we defined $f_n$ in section $1.3$ and set, as in there, $h_n = \frac{1}{2n+d}$. 
For any $n > 0$, $A \in \Gamma_d$, we define $f_n^A = T_{A,h_n}f_n$. 

\begin{lem}
\label{areeigen}
For any $n > 0$, $(-h_n^2\Delta+|x|^2)f_n^A = f_n^A$.
\end{lem}

\begin{proof}
Observe that $p \circ A=p$, thus $$f_n^A=T_{A,h_n}(p \circ A)^w(x,h_nD)f_n=T_{A,h_n}T_{A,h_n}^*p^w(x,h_nD)T_{A,h_n}f_n = p^w(x,h_nD)f_n^A.$$
\end{proof}

\noindent
Let $\mathcal{C}_1 = \{|x|^2+|\xi|^2=x_1^2+\xi_1^2=1\}$. Recall from proposition \ref{oneorbitdim} that for any $a \in \mathscr{C}^{\infty}_c(\R^{2d})$, 
$$\langle a^w(x,h_nD)f_n,\,f_n\rangle \longrightarrow \frac{1}{2\pi}\int_{\mathcal{C}_1}{a\,ds}.$$

\begin{prop}
Let $A \in \Gamma_d$. Then, for any $a \in \mathscr{C}^{\infty}_c(\R^d)$, as $n \longrightarrow \infty$:
$$\langle a^w(x,h_nD)f_n^A,\,f_n^A\rangle \longrightarrow \frac{1}{2\pi}\int_{A\mathcal{C}_1}{a\,ds}.$$
\end{prop}

\begin{proof}
\begin{align*}
\langle a^w(x,h_nD)f_n^A,\,f_n^A\rangle &= \langle a^w(x,h_nD)T_{A,h_n}f_n,\,T_{A,h_n}f_n\rangle\\
&= \langle T_{A,h_n}^*a^w(x,h_nD)T_{A,h_n}f_n,\,f_n\rangle \\
&= \langle (a \circ A)^w(x,h_nD)f_n,\,f_n\rangle \\
&\longrightarrow \frac{1}{2\pi}\int_{\mathcal{C}_1}{a\circ A\,ds} = \frac{1}{2\pi}\int_{A\mathcal{C}_1}{a\,ds}
\end{align*}
\end{proof}

\section{End of the proof}

\subsection{Convex measures}~\\
Recall that we defined a convex measure as a convex combination of elementary measures, elementary measures being invariant measures that are supported along a single orbit. 

We will need one lemma to prove the theorem for convex measures:

\begin{lem}
Let $b \in S(1)$, $A \in \Gamma_d$. Assume that $b$ is zero near $A\mathcal{C}_1$. Then $b^w(x,h_nD)f_n^A = o_{L^2}(1)$. 
\end{lem}

\begin{proof}
Using \ref{micro} and \ref{areeigen}, we may assume that $b$ is compactly supported. Then, using \ref{compo_S1} and \ref{bounded}
\begin{align*}
\|b^w(x,h_nD)f_n^A\|^2_{L^2} &= \langle b^w(x,h_nD)f_n^A,\,b^w(x,h_nD)f_n^A\rangle\\
&= \langle \overline{b}^w(x,h_nD)b^w(x,h_nD)f_n^A,\,f_n^A\rangle\\
&= \langle (\left(\overline{b}b+O_{S(1)}(h_n)\right)^w(x,h_nD)f_n^A,\,f_n^A\rangle\\
&= O(h_n) + \langle (|b|^2)^w(x,h_nD)f_n^A,\,f_n^A\rangle\\
&= O(h_n)+o(1)+\frac{1}{2\pi}\int_{A\mathcal{C}_1}{|b|^2\,ds} \longrightarrow 0
\end{align*}
\end{proof}

\begin{cor}
\label{mixedterms}
Let $c \in S(1)$. Let $A,B \in \Gamma_d$ be such that $\mathcal{A} = A\mathcal{C}_1$ and $\mathcal{B} = B\mathcal{C}_1$ are distinct (hence disjoint) orbits. Then, as $n \longrightarrow \infty$, $$\langle c^w(x,h_nD)f_n^A,\,f_n^B\rangle \longrightarrow 0.$$
\end{cor}

\begin{proof}
There exists real-valued $a,b \in S(1)$ that are $1$ near $\mathcal{A}$ and $\mathcal{B}$ and with disjoint supports. Then, in $L^2$, $f_n^A = a^w(x,h_nD)f_n^A+o(1)$ and $f_n^B = b^w(x,h_nD)f_n^B+o(1)$. Hence $\langle c^w(x,h_nD)f_n^A,\,f_n^B\rangle = o(1) + \langle b^w(x,h_nD)c^w(x,h_nD)a^w(x,h_nD)f_n^A,\,f_n^B\rangle$, because $b^w(x,h_nD)$ is self-adjoint and because of \ref{bounded}. From \ref{compo_S1}, $b^w(x,hD)c^w(x,hD)a^w(x,hD) = (bca)^w(x,hD) + O_{S(1)}(h) = o_{S(1)}(1)$. Using \ref{bounded} ends the proof. 
\end{proof}

\begin{prop}
Let $\mu$ be a convex measure on $S^{2d-1}$. For each $n$, there exists an eigenfunction $g_n$ of the quantum harmonic oscillator (with parameter $h_n$) (with eigenvalue $1$, and unit $L^2$ norm) such that 
$$\forall a \in \mathscr{C}^{\infty}_c(\R^{2d}),\, \langle a^w(x,h_nD)g_n,\,g_n\rangle \longrightarrow \int_{S^{2d-1}}{a\,d\mu}.$$
\end{prop}

\begin{proof}
We can write $\mu = \sum_{i=1}^p{\lambda_ic_i}$, where $\lambda_i$ are non-negative real numbers adding up to $1$, and for any $a \in \mathscr{C}^0(S^{2d-1})$, $\int_{S^{2d-1}}{a\,dc_i} = \frac{1}{2\pi}\int_{\mathcal{C}_i}{f\,ds}$, where $\mathcal{C}_i$ are distinct orbits. 

Let, for $1 \leq i \leq p$, $A_i \in \Gamma_d$ be such that $\mathcal{C}_i = A_i\mathcal{C}$. Let $g_n^{(1)} = \sum_{i=1}^p{\sqrt{\lambda_i}f_n^{A_i}}$.

Then corollary \ref{mixedterms} proves that for any $a \in \mathscr{C}^{\infty}_c(\R^{2d})$, $$\langle a^w(x,h_nD)g_n^{(1)}\,,g_n^{(1)}\rangle \longrightarrow \int_{S^{2d-1}}{ad\mu}, \|g_n^{(1)}\|_{L^2} = 1+o(1).$$ Now take $g_n = \|g_n^{(1)}\|_{L^2}^{-1}g_n^{(1)}$ for each $n$. 
\end{proof}

\subsection{How to get all measures}~\\
Let us begin with the following theorem:

\begin{prop}
The set of convex measures is dense in the set of probability measures on $S^{2d-1}$ that are invariant with respect to the Hamilton flow, in the weak-* topology.
\end{prop}

\begin{proof}
Let $V$ be the locally convex topological vector space of all measures on $S^{2d-1}$, endowed with the weak-* convergence topology. Let $K \subset V$ be the set of all invariant probability measures on $S^{2d-1}$. $K$ is convex, closed (it can be defined by a set of equations involving only continuous functions on $S^{2d-1}$) and contained in the unit ball of $V$ ($V$ is the dual space of $\mathscr{C}(S^{2d-1})$ which is a Banach space so has a natural norm), which is compact in the weak-* topology (Banach-Alaoglu theorem), so $K$ is convex and compact. \\
To prove the claim, it is enough to show that the extremal points of $K$ (as a convex set) are the elementary measures (then we apply the Krein-Milman theorem). \\
Now, let $\mu \in K$ be extremal, assume it is not elemnetary. Then we can find $x,y \in S^{2d-1}$ in the support of $\mu$ that do not belong to the same orbit. So we can take some small open set $U \subset S^{2d-1}$ containing $x$, such that for every $x' \in \overline{U}$, $x'$ and $y$ do not belong to the same orbit. Let $X$ be the union of all orbits intersecting $\overline{U}$, and $Y = S^{2d-1} \backslash X$. Let $\mu_1 = \frac{\mu(\bullet \cap X)}{\mu(X)}$, $\mu_2 = \frac{\mu(\bullet \cap Y)}{\mu(Y)}$, then $\mu = (1-\mu(X))\mu_2 + \mu(X)\mu_1$, and $0 < \mu(X) < 1$, and $\mu_1$ and $\mu_2$ are in $K$ and distinct. Thus, the support of $\mu$ is a single orbit $\mathcal{C}_0$, and since $\mu$ is invariant under the Hamilton flow, $\mu$ is the corresponding elementary measure (that is $\mu:\,a \longmapsto \frac{1}{2\pi}\int_{\mathcal{C}_1}{a\,ds}$), which proves the claim. 
\end{proof}

~\\

Roughly speaking, we would like to argue as follows:
\begin{enumerate}
\item All convex measures are limits of eigenvectors.
\item All suitable measures are limits of convex measures.
\item Therefore, all suitable measures are limits of eigenvectors.
\end{enumerate}

~\\

As stated, this reasoning is not correct, as shown by the following example: it is known that $1_{\Q}$ is not a pointwise limit of continuous functions, see for instance \cite[Theorem 5]{baire}. However, let $f_n = 1_{\Z/n!}$. It is easy to see that $f_n \longrightarrow 1_{\Q}$ pointwise, and that $f_n(x) = \lim\limits_{p \rightarrow \infty}e^{-p|\sin(n!x\pi)|}$, so $f_n$ is a pointwise limit of continuous functions. 

~\\
However, here is a fix: we will now prove that there exists a metric space $X$, containing all the linear forms induced by the eigenfunctions we have used, and all of the limits we are interested in, such that the metric over $X$ defines the same topology as the weak-* convergence. 

Recall from \ref{bounded} that, for any $u \in L^2(\R^n)$ with norm $1$, any invariant probability measure $\mu$ on $S^{2d-1}$, any $0 < h < 1$, any $a \in \mathscr{C}^{\infty}_c{\R^{2d}}$, 
\begin{align*}
|\langle a^w(x,hD)u\,,u\rangle| &\leq C\|a\|_{\mathscr{C}^{Md}_b(\R^{2d})}\\
\left|\int_{S^{2d-1}}{a\,d\mu}\right| &\leq C\|a\|_{\mathscr{C}^{Md}_b(\R^{2d})},
\end{align*}
where $C,M \geq 1$ depend only on $d$. So we set $X$ to be the ball of $B^*$ with center $0$ and radius $C$, where $B$ is the completion in the $\mathscr{C}^{Md}_b$ norm of $\mathscr{C}^{\infty}_c(\R^{2d})$, that is, $B$ is the Banach space of all $\mathscr{C}^{Md}$ functions defined over $\R^{2d}$ such that all their derivatives of order at most $Md$ go to zero at infinity. 

The space $X$, with the weak-* topology (we denote it $X_*$ as a topological space), is Hausdorff compact (Banach-Alaoglu theorem, see for instance \cite[Theorem 3.16]{HB}). Assume now that $B$ is separable. Let $(a_n) \in B^{\N}$ be a sequence of functions with norm $1$ that is dense in the unit sphere. Let $$\delta: (l,m) \in X^2 \longmapsto \sum_{n \in \N}{\min(|(l-m)(a_n)|,2^{-n})}.$$

It is easy to see that $\delta$ is a metric on $X$, and that $\Id: X_* \longrightarrow (X,\delta)$ is a bijective continuous map, therefore it is a homeomorphism and $X_*$ is a metric space.  

Now, it remains to prove that $B$ is separable. We will show that the set of test function is separable for the (stronger) test function topology and dense in $B$. The density part is well-known, for instance with \cite[Cor.2 p.159]{Trv}. 
For the separability of the set of test function, we consider a sequence $K_p$ of compact subsets of $\R^d$, such that $K_p \subset \overset{\circ}{K_{p+1}}$ and $\R^d = \bigcup_p{K_p}$. It is enough to show that for each $p$, the space $\mathscr{D}(K_p)$ of test functions with support in $K_p$ is separable. 
Now, there is a natural topology-preserving embedding $$f \in \mathscr{D}(K_p) \longmapsto \left(\partial_{\alpha}f\right)_{\alpha \in \N^d} \in \mathscr{C}^0(K_p)^{\N^d}.$$ 
Now, from \cite[Theorem 6.6]{Con}, $\mathscr{C}^0(K_p)$ is a separable metric space; from \cite[Theorems VIII-6.2, VIII-7.3, IX-5.6]{Du}, it is second-countable, so is $\mathscr{C}^0(K_p)^{\N^d}$, thus the image of the embedding is separable, and therefore $\mathscr{D}(K_p)$ is separable, which ends the proof. 

~\\

~\\

\end{document}